\documentclass[12pt]{amsart}
\usepackage{amsmath,amssymb,amsfonts,amsthm}
\usepackage{graphicx}
\usepackage{enumerate}
\usepackage{tikz}
\usepackage[mathscr]{eucal}
\usepackage{float}
\numberwithin{equation}{section}
\newtheorem{theorem}{Theorem}[section]
\newtheorem{definition}{Definition}[section]
\newtheorem{lemma}[theorem]{Lemma}
\newtheorem{proposition}[theorem]{Proposition}
\newtheorem{corollary}[theorem]{Corollary}
\newtheorem{note}[theorem]{Note}
\newtheorem{claim}[theorem]{Claim}
\newtheorem{conjecture}[theorem]{Conjecture}
\numberwithin{equation}{section}

 \makeatletter
      \def\@setcopyright{}
      \def\serieslogo@{}
      \makeatother

\begin{document}
\bibliographystyle{plain}

\title{Marcello's completion of graphs}
\author{J. Kok} 
\address {Independent Mathematics Researcher, City of Tshwane, South Africa.}
\email{jacotype@gmail.com}
\date{}

\maketitle

\begin{abstract}
This paper initiates a study on a new optimization problem with regards to graph completion. The defined procedure is called, \emph{Marcello's completion} of a graph. For graph $G$ of order $n$ the \emph{Marcello number} is obtained by iteratively constructing graphs, $G_1,G_2,\dots,G_k$ by adding a maximal number of edges between pairs of distinct, non-adjacent vertices in accordance with the \emph{Marcello rule}. If for smallest $k$ the resultant graph $G_k \cong K_n$ then the Marcello number of a graph $G$ denoted by $\varpi(G)$ is equal to $\varpi(G) = k$. By convention $\varpi(K_n) = 0$, $n \geq 1$. Certain introductory results are presented.
\end{abstract}
\section{Introduction}
For general notation and concepts in graphs see \cite{bondy, harary, west}. Throughout the study only finite, simple and undirected graphs $G$ of order $n \geq 3$ and size $\varepsilon(G) \geq 1$ will be considered. A graph $G$ will have vertex set $V(G) = \{v_i: i = 1,2,3,\dots,n\}$. Recall that an edge $v_iv_j \in E(G)$ represents an unordered pair of vertices i.e. $\{v_i,v_j\}$, $v_i,v_j \in V(G)$. An edge implies some acquaintance between the vertices. The rule of acquaintance is called the incidence function. Since $G$ is required to be simple it follows that vertex $v_i$ is distinct from vertex $v_j$ and other than $v_iv_j \equiv v_jv_i$ the edge $v_iv_j$ is not repeated in $G$.

The addition or deletion of edges in a graph $G$ and the effect thereof have been widely studied. For example, see \cite{faudree, heggernes, majeed, yamuna}. Researching \emph{mathematics for the sake of mathematics} is a sufficient motivation to study \emph{Marcello's completion}\footnote{See dedication for an explanation.} of a graph. In a more general sense the study of the \emph{degree of separation} in social networks can benefit from the study of Marcello's completion. Social networks are interested in ways to achieve the ideal degree of separation namely, $1$. In graph theoretical terms a social network has a representative graph structure say, $G$. An edge $v_iv_j$ indicates that there is acquaintance between members $v_i$ and $v_j$. A theory exists that each vertex $v_i$ (member of network) has eccentricity (or degree of separation), $\epsilon_G(v_i) \leq 6$. Recent research of particular social network platforms shows that for some, $\epsilon_G(v_i) \leq 4$ or put differently, $diam(G) \leq 4$. The notion of Marcello's completion could lead to real world strategies to reduce the upper bound of the degree of separation. 

Let $G = (V(G),E(G))$ be a graph of order $n \geq 1$. Recall that the complement of $G$ denoted by $\overline{G}$ has the vertex set $V(G)$ and the edge set $E(\overline{G}) = \{v_iv_j:$~ if and only if $v_iv_j \notin E(G)\}$. Classical graph completion can be obtained by the rule: Consider $G = (V(G),E(G))$ and derive $E(\overline{G})$. Then construct the edge induced graph $\langle E(G) \cup E(\overline{G})\rangle$. It is easy to see that if the \emph{complement completion rule i.e. derive $E(\overline{G})$ then construct $\langle E(G) \cup E(\overline{G})\rangle$} is considered as a global iteration on $G$  then classical graph completion requires only one global iteration to reach completeness. Note that the global iteration consists of several local iterations. Other graph completion rules may require more than one global iteration to reach completeness. To complicate matters some rules can result in different outcomes (or non-isomorphic resultant graphs). Herein lies the endeavor to optimize the application of the graph completion rule. The next section introduces the notion of Marcello's completion of a graph $G$.
 
\section{Marcello's completion of a graph}
Consider a non-complete graph $G = (V(G), E(G))$ of order $n \geq 3$ with $E(G)$ a non-empty set. Note that Definition \ref{def1} below provides for a single global iteration to obtain $G_1$ from $G_0 = G$. With repeated applications it is possible to obtain $G_i$ from $G_{i-1}$, $i = 2,3,4,\dots$ Imbedded in a single global iteration a number of local iterations are performed which consider a subset of the vertices of the applicable graph. The final global iteration which renders a complete graph terminates the repetition of the iterative procedure. Note that a repetition of Definition \ref{def1} is called a Marcello global iteration applied to the applicable $G_i$, $i = 0,1,2,\dots$ resulting in the graph $G_{i+1}$. Assume that from $G_0 = G$ the graphs $G_1,\dots,G_{i-1}$ are constructed where $G_{i-1} \neq K_n$.
\begin{definition}\label{def1}
Let $H_{i-1} = G_{i-1}$. Let $X_{i-1} \subseteq V(G_{i-1})$ where $X_{i-1} = \{v_j: 0 < deg_{G_{i-1}}(v_j) < n-1\}$ and $\ell = |X_{i-1}|$. 
 
For $k = 1,2,3,\dots, \ell$ go to Step 1: 

Step 1:  Let $j = k$ and select any vertex $v_{t_j} \in X_{i-1}$, ($t_j = s$ for some $v_s \in X_{i-1}$). Go to Step 2.

Step 2: If possible, add at most $deg_{G_{i-1}}(v_{t_j})$ edges (called, Marcello edges) to non-neighbors of $v_{t_j}$ in $H_{j-1}$. Label the resultant graph $H_j$. If $H_j$ is complete then label it as $G_i$ and exit ($G$ has been completed). Else, $X_{i-1}\mapsto X_{i-1}\backslash \{v_{t_j}\}$. Go to step 3.

Step 3: If $X_{i-1} = \emptyset$ or if all vertices remaining, say $v_q \in X_{i-1}$ has $deg_{H_j}(v_q) = n-1$ then label it as $G_i$ and exit (a global iteration completed). Else, consider the next $k$ and go to Step 1 (a local iteration completed).\\
\end{definition}
In general, if the vertex $v_i$ is under consideration (a local iteration) and $v_i$ accepts a non-adjacent vertex $v_j$ as a neighbor then the corresponding Marcello edge is denoted by $v^+_iv_j$. Although the notation suggests pseudo orientation the \lq orientation\rq  will not be depicted and the Marcello edges will be dashed. Since the number of graphs of order $n$ is finite, a finite number of global iterations will result in a complete graph. The minimum number of global iterations required to reach a complete graph $K_n$ is called the \emph{Marcello number} of the graph $G$ and is denoted by $\varpi(G)$. A sequence of graphs, say $G_1$,$G_2$,$G_3$,$\cdots, G_k$ is called a \emph{Marcello sequence} of $G$.
\begin{claim}\label{clm1}
If graph $G$ has $\varpi(G) = k$ then firstly, a Marcello sequence of graphs, say $G_1$,$G_2$,$G_3$,$\cdots, G_k$ exists such that,
\begin{center}
$\varpi(G_1) = k-1$, $\varpi(G_2) = k-2$, $\varpi(G_3) = k-3$, $\cdots, \varpi(G_k) = 0$.
\end{center}
Secondly, a Marcello sequence of a graph is not necessarily unique. Two or more sequences of a graph $G$, say
\begin{center}
$G_1$,$G_2$,$G_3$,$\cdots, G_k$ and $G^*_1$,$G^*_2$,$G^*_3$,$\cdots, G^*_k$
\end{center}
may exist which result in completion such that there exists at least one pair $G_i \ncong G^*_i$. 
\end{claim}
From Claim \ref{clm1} it follows naturally that $\varpi(K_n) = 0$, $n \geq 1$. For a null graph (or edgeless graph) of order $n \geq 2$ and denoted by $\mathfrak{N}_n$ we define $\varpi(\mathfrak{N}_n) = \infty$. Since $K_1$ can be considered as complete or null the unambiguous lower bound $n \geq 2$ is required for the null graph.

Consider the path $P_6$ in Figure \ref{fig1}. Assume that Marcello's completion runs through the local iterations by sequentially considering the vertices $v_2, v_4, v_3, v_5, v_6, v_1$. Although not unique, a maximum number of permissible Marcello edges to be added are $v^+_2v_4$, $v^+_2v_6$; $v^+_4v_1$, $v^+_4v_6$; $v^+_3v_5$, $v^+_3v_6$; $v^+_5v_1$, $v^+_5v_2$; $v^+_6v_1$; $v^+_1v_3$. See Figure \ref{fig2}. It is worthy to note that if a graph $G$ has $\varpi(G) = 1$ not all applications of Definition \ref{def1} will necessarily lead to completeness. For example, with regards to Figure \ref{fig1} although not unique, a maximal number of permissible Marcello edges to be added could be $v^+_1v_4$, $v^+_2v_4$; $v^+_2v_5$; $v^+_3v_5$, $v^+_3v_6$; $v^+_4v_6$; $v^+_5v_1$; $v^+_6v_2$. See Figure \ref{fig3}.

\begin{figure}[h]
\centering
\begin{tikzpicture}[scale=0.6]
\tikzstyle{every node}=[draw,shape=circle];
\node (v0) at (0:3) [] {$v_1$};
\node (v1) at (60:3) [] {$v_2$};
\node (v2) at (2*60:3) [] {$v_3$};
\node (v3) at (3*60:3) [] {$v_4$};
\node (v4) at (4*60:3) [] {$v_5$};
\node (v5) at (5*60:3) [] {$v_6$};
\draw (v0) -- (v1)
(v0) -- (v1)
(v1) -- (v2)
(v2) -- (v3)
(v3) -- (v4)
(v4) -- (v5)
;
\end{tikzpicture}
\caption{Graph $P_6$ for which $\varpi(P_6) = 1.$}\label{fig1}
\end{figure}
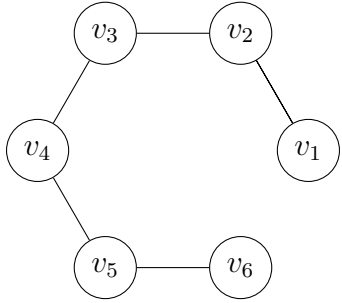

\begin{figure}[h]
\centering
\begin{tikzpicture}[scale=0.6]
\tikzstyle{every node}=[draw,shape=circle];
\node (v0) at (0:3) [] {$v_1$};
\node (v1) at (60:3) [] {$v_2$};
\node (v2) at (2*60:3) [] {$v_3$};
\node (v3) at (3*60:3) [] {$v_4$};
\node (v4) at (4*60:3) [] {$v_5$};
\node (v5) at (5*60:3) [] {$v_6$};
\draw (v0) -- (v1)
(v1) -- (v2)
(v2) -- (v3)
(v3) -- (v4)
(v4) -- (v5)
;
\draw[dashed] (v3) -- (v5)
(v0) -- (v2)
(v0) -- (v3)
(v0) -- (v4)
(v0) -- (v5)
(v1) -- (v3)
(v1) -- (v4)
(v1) -- (v5)
(v2) -- (v4)
(v2) -- (v5)
;
\end{tikzpicture}
\caption{Graph $P_6$ for which one global iteration yields $K_6$.}\label{fig2}
\end{figure}
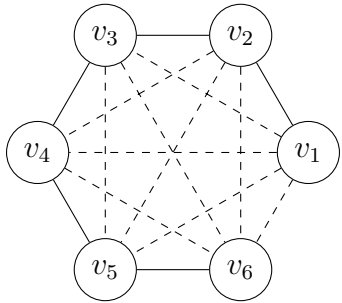

\begin{figure}[h]
\centering
\begin{tikzpicture}[scale=0.6]
\tikzstyle{every node}=[draw,shape=circle];
\node (v0) at (0:3) [] {$v_1$};
\node (v1) at (60:3) [] {$v_2$};
\node (v2) at (2*60:3) [] {$v_3$};
\node (v3) at (3*60:3) [] {$v_4$};
\node (v4) at (4*60:3) [] {$v_5$};
\node (v5) at (5*60:3) [] {$v_6$};
\draw (v0) -- (v1)
(v1) -- (v2)
(v2) -- (v3)
(v3) -- (v4)
(v4) -- (v5)
;
\draw[dashed] (v3) -- (v5)
(v0) -- (v3)
(v0) -- (v4)
(v1) -- (v3)
(v1) -- (v4)
(v1) -- (v5)
(v2) -- (v4)
(v2) -- (v5)
;
\end{tikzpicture}
\caption{Graph $P_6$ for which one global iteration did not yield $K_6$.}\label{fig3}
\end{figure}
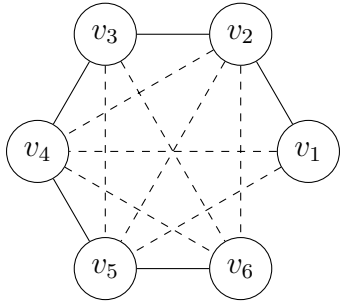

The above brings us to an important observation.
\begin{note}
Definition \ref{def1} is defined such that termination (exiting) is guaranteed. The aforesaid is valid because the order of the graph is finite. However, Marcello's completion is not well-defined. The aforesaid is true because the outcome of a global iteration is not necessarily unique. The aforesaid brings us to an somewhat obscure theorem.
\end{note}
\begin{theorem}\label{thm1}
Let $G$ be a non-complete and non-null graph of order $n \geq 3$. If $\varpi(G) = 1$ then
\begin{center}
$\sum\limits_{0<deg_G(v_i) < n-1}deg_G(v_i) + |E(G)| \geq \frac{n(n-1)}{2}$.
\end{center}
\end{theorem}
\begin{proof}
Firstly the bound $n \geq 3$ is trivially clear. Furthermore, the maximum number of edges that can by added through one global iteration of Definition \ref{def1} is given by,
\begin{center}
$\sum\limits_{0<deg_G(v_i) < n-1}deg_G(v_i)$.
\end{center}
Since $\varpi(G) = 1$ it implies completeness has been reached. Hence, $G_1 \cong K_n$. The latter settles the result since $|E(G)| < \frac{n(n-1)}{2}$.
\end{proof}
\begin{corollary}\label{cor1}
If $\sum\limits_{0<deg_G(v_i) < n-1}deg_G(v_i) + |E(G)| < \frac{n(n-1)}{2}$ then
\begin{center}
$\varpi(G) \geq 2$.
\end{center}
\end{corollary}
Observe that if $deg_G(v_i) < n-1$ $\forall~i$ then the left-hand side of each inequality in Theorem \ref{thm1} and Corollary \ref{cor1} can be replaced by $3|E(G)|$. The converse of Theorem \ref{thm1} and Corollary \ref{cor1} are not always valid. Consider the graph below.
\begin{figure}[h]
\centering
\begin{tikzpicture}[scale=0.6]
\tikzstyle{every node}=[draw,shape=circle];
\node (v0) at (0:3) [] {$v_1$};
\node (v1) at (60:3) [] {$v_2$};
\node (v2) at (2*60:3) [] {$v_3$};
\node (v3) at (3*60:3) [] {$v_4$};
\draw (v0) -- (v1)
(v1) -- (v2)
;
\end{tikzpicture}
\caption{Graph for which one global iteration cannot yield completeness.}\label{fig4}
\end{figure}
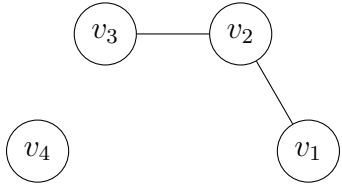

Despite that fact that
\begin{center}
$\sum\limits_{0<deg_G(v_i) < n-1}deg_G(v_i) + |E(G)| = 6 \geq \frac{n(n-1)}{2} = 6$,
\end{center}
the graph cannot be completed in one global iteration. Up to equivalence of options (or up to isomorphism of resultant graphs) only two options avail during the first global iteration. The options are either $v^+_1v_3$, $v^+_2v_4$, $v^+_3v_4$ or $v^+_1v_4$, $v^+_2v_4$, $v^+_3v_4$. In both cases the resultant $G_1$$^{,s}$ are non-complete. See Figures 5 and 6 below.

\begin{figure}[h]
\centering
\begin{tikzpicture}[scale=0.6]
\tikzstyle{every node}=[draw,shape=circle];
\node (v0) at (0:3) [] {$v_1$};
\node (v1) at (60:3) [] {$v_2$};
\node (v2) at (2*60:3) [] {$v_3$};
\node (v3) at (3*60:3) [] {$v_4$};
\draw (v0) -- (v1)
(v1) -- (v2)
;
\draw[dashed] (v0) -- (v2)
(v1) -- (v3)
(v2) -- (v3)
;
\end{tikzpicture}
\caption{}
\end{figure}

\begin{figure}[h]
\centering
\begin{tikzpicture}[scale=0.6]
\tikzstyle{every node}=[draw,shape=circle];
\node (v0) at (0:3) [] {$v_1$};
\node (v1) at (60:3) [] {$v_2$};
\node (v2) at (2*60:3) [] {$v_3$};
\node (v3) at (3*60:3) [] {$v_4$};
\draw (v0) -- (v1)
(v1) -- (v2)
;
\draw[dashed] (v0) -- (v3)
(v1) -- (v3)
(v2) -- (v3)
;
\end{tikzpicture}
\caption{}
\end{figure}

Henceforth, all graphs $G$ will be non-complete and non-null to avoid ambiguity. The aforesaid does not exclude the possibility that a graph may be disconnected and has a null graph component. So a graph $G = H \cup \mathfrak{N}_n$ is permitted provided that $H$ is connected. The graph $H$ may be complete because $G$ remains non-complete. Two self-evident claims will be useful.
\begin{claim}\label{clm2}
If $\varpi(H) = k$ and $H$ is a spanning subgraph of $G$ then $\varpi(G) \leq k$.
\end{claim}
\textit{Motivation:} Graph $G$ has the possible advantage of having edges not in $H$.
\begin{claim}\label{clm3}
Consider graph $G$ of order $n$. Let $V^* = \{v_i: deg_G(v_i) = n-1\}$. Consider the induced subgraph $H = \langle V(G)\backslash V^*\rangle$. If $\varpi(H) = k$ then $\varpi(G) \leq k$.
\end{claim}
\textit{Motivation:} With regards to vertex degrees we have that $deg_G(v_i) \geq deg_H(v_i)$, for all vertices common to $G$ and $H$.
\begin{claim}\label{clm4}
Consider graph $G$ of order $n$. If $G$ has two or more isolated vertices then $\varpi(G) \geq 2$.
\end{claim}
\textit{Motivation:} On the first global iteration an isolated vertex $v_i$ cannot yield an edge $v^+_iv_j$. See the Marcello rule.
\begin{claim}\label{clm5}
Consider graph $G$ of order $n$. If $G$ has four or more pendant vertices then $\varpi(G) \geq 2$.
\end{claim}
\textit{Motivation:} On the first global iteration $t \geq 4$ pendant vertices can at best yield a cycle $C_t$. See the Marcello rule.

\subsection{The two extremes:} The extreme graphs $K_n$ and $\mathfrak{N}_m$ will be considered. Recall that the join $G+H$ of two graphs $G$ and $H$ is obtained by joining each vertex of $G$ to each vertex of $H$.
\begin{theorem}\label{thm2}
Let $G = K_n\cup K_m$, $n \geq 2$, $m \geq 1$. Then $\varpi(G) = 1$.
\end{theorem}
\begin{proof}
Let $V(K_n) = \{v_i: i = 1,2,3,\dots,n\}$ and $V(K_m) = \{u_j:j = 1,2,3,\dots,m\}$. Without loss of generality assume $n \geq m$.

Case 1: Let $n > m$. For this case $n-m \geq 1$. By Definition \ref{def1} the following edges can be added on the first global iteration:
\begin{center}
$v^+_1u_j$ for $j = 1,2,3,\dots,m$;

$v^+_2u_j$ for $j = 1,2,3,\dots,m$;

$v^+_3u_j$ for $j = 1,2,3,\dots,m$;

$\vdots$

$v^+_nu_j$ for $j = 1,2,3,\dots,m$.
\end{center}
Clearly, $G_1$ is complete. This settles the result for $n > m$.

Case 2(a): Let $n = m$ and even. By Definition \ref{def1} the edges which can be added on the first global iteration to exhaust the vertices $v_i$, $ i = 1,2,3,\dots,n$ are: 
\begin{center}
$v^+_1u_j$ for $j = 1,2,3,\dots,m-1$;

$v^+_2u_j$ for $j = 2,3,4,\dots,m$;

$v^+_3u_j$ for $j = 1,2,3,\dots,m-1$;

$v^+_4u_j$ for $j = 2,3,4,\dots,m$;

$\vdots$

$v^+_{n-1}u_j$ for $j = 1,2,3,\dots,m-1$;

$v^+_nu_j$ for $j = 2,3,4,\dots,m$.
\end{center}
Clearly, the vertices $u_1,u_m \in V(K_m)$ have not been \lq completed\rq. Each of these vertices may initiate $m-1$ Marcello edges and exactly $\frac{m}{2}$ Marcello edges are required. We determine for which values of $m$ we have $m-1 \geq \frac{m}{2}$. The latter is valid for $m \geq 2$. Hence, Marcello completion is possible in one global iteration. Therefore, $\varpi(G) = 1$.

Case 2(b): For $n = m$ and odd, the proof follows by similar reasoning as that found in the proof of Case 2(a).
\end{proof}
A corollary follows. The proof thereof follows the framework of reasoning found in the proof of Theorem \ref{thm2}. The proof is left for the reader.
\begin{corollary}\label{cor2}
For $G \cup H$ it follows that
\begin{center}
$\varpi(G \cup H) \leq max\{\varpi(G)+1,\varpi(H)+1\}$.
\end{center}
\end{corollary} 
\begin{proposition}\label{prop2}
Let $G = \mathfrak{N}_m + \mathfrak{N}_n$, $n \geq m \geq 1$. Then $\varpi(G) = 1$ if and only if $(m,n) \neq (1,1)$ and $ n \leq 2m+1$. 
\end{proposition}
\begin{proof}
Part (a): Let $1 \leq m \leq n \leq 2m+1$. Since $G \cong K_{m,n}$ we consider two sub-cases. 

Sub-case (a)(i): Consider $G = K_{n,n}$, $n = m \geq 2$. It is easy to see that the result holds for $K_{2,2}$ and $K_{3,3}$. Assume it holds for $n = t$. By assumption $\varpi(K_{t,t}) = 1$. Note that each vertex degree in $K_{t,t}$ equals $t$. Consider $K_{t+1,t+1}$ and note that each vertex degree equals $t+1$. Thus, by immediate induction the result $\varpi(G) = 1$ holds for $g = K_{t+1,t+1}$. That settles the result $\forall~ n = m \geq 2$.

Sub-case (a)(ii): Consider $G = K_{m,n}$, $1 \leq m < n \leq 2m+1$. Clearly, for $m = 1$ the graphs $K_{1,2}$ and $K_{1,3}$ can be completed in one global iteration. Note that $K_{1,4}$ cannot be completed in one global iteration. Let the independent $n$ vertices be labeled $v_1$, $v_2$, $v_3$,\dots, $v_n$. It is equally easy to see that $K_{2,3}$, $K_{2,4}$ and $K_{2,5}$ can be completed in one global iteration. Observe that in the Marcello completion of $K_{2,3}$ and without loss of generality by, say $v^+_1v_2$, $v^+_1v_3$, $v^+_2v_3$ a total of $3$ degree-counts (one at $v_2$ and two at $v_3$) were not utilized. Similarly for the completion of $K_{2,4}$ one degree-count at each vertex $v_3$ and $v_4$ was not utilized. The aforesaid degree-counts together with the degree of $v_5$ in $K_{2,5}$ make it possible to complete $K_{2,5}$ in one global iteration. Note that $K_{2,6}$ cannot be completed in one global iteration. Hence, the result holds for $m = 1$ and $n = 2,3$; $m = 2$ and $n = 3,4,5$ respectively. Assume it holds for $6 \leq n = t$ and $t$ is even. For $m \geq \frac{t}{2}$ consider the sufficient case $m = \frac{t}{2}$. Hence, in $K_{m,t}$ each vertex $v_i$, $1 \leq i \leq t$ has degree equal to $\frac{t}{2}$. By the Marcello completion pattern from the cases above i.e. 
\begin{center}
$v^+_1v_j$ for $j = 2,3,\dots,\frac{t+2}{2}$;

$v^+_2u_j$ for $j = 3,4,\dots,\frac{t+4}{2}$;

$v^+_3u_j$ for $j = 4,5,\dots,\frac{t+6}{2}$;

$\vdots$

$v^+_{\frac{t+4}{2}}u_j$ for $j = 1,\frac{t+6}{2},\dots,t$;

$v^+_{\frac{t+6}{2}}u_j$ for $j = 1,2,\frac{t+8}{2},\dots,t$;

$\vdots$

$v^+_tu_j$ for $j = 1,2,3,4,\dots,\frac{t-2}{2}$;
\end{center} 
it follows that exactly $\frac{t-2}{2}$ degree-counts were not utilized. Firstly, by immediate induction the not-utilized degree-counts follows in similar fashion as for all even $t$. Secondly, by a similar completion pattern from the $t$ is even cases the odd case $n = t+1$ can be completed with the available degree of vertex $v_{t+1}$ namely, $\frac{t}{2}$ with the added $\frac{t-2}{2}$ degree-counts read together with Corollary \ref{cor1}. By induction the result follows for $n = t+1$.

If the proof begins with $n = t$ and $t$ is odd then similar reasoning settles the result. Hence, the result holds for $m \geq 3$ and $m + 1 \leq n \leq 2m+1$.

Part (b): Let $\varpi(G) = 1$. By assumption Marcello completion applies to non-complete graphs hence, $(m,n) \neq (1,1)$. Let $m < n$ and $n > 2m+1$. We only have to consider the vertex subset $\{v_1,v_2,v_3,\dots, v_n\}$ as an empty subgraph, say $H$ with each vertex allocated the artificial degree equal to $m$. It follows that since $n > 2m+1$ we have $m < \frac{n-1}{2}$. For subgraph $H$ it follows that $n \times m < \frac{n(n-1)}{2}$. Also $\varepsilon(H) = 0$ thus
\begin{center}
$0 + nm < \frac{n(n-1)}{2}$.
\end{center}
Hence, by Theorem \ref{thm1} it follows that Marcello completion of $H$ is not possible in one global iteration. Therefore, the Marcello completion of $G$ is not possible in one global iteration. The aforesaid is a contradiction. Thus, $1 \leq m \leq n \leq 2m+1$.
\end{proof}
A corollary follows. The proof thereof follows the framework of reasoning found in the proof of Proposition \ref{prop2}. The proof is are left for the reader.
\begin{corollary}\label{cor3}
For $G + H$ it follows that
\begin{center}
$\varpi(G + H) \leq max\{\varpi(G),\varpi(H)\}$.
\end{center}
\end{corollary} 
Two existence theorems follow.
\begin{theorem}\label{thm3}
For $\mathfrak{N}_m$, $m \geq 2$ there exists a minimum $n \geq 1$ such that $\varpi(K_n + \mathfrak{N}_m) = 1$. 
\end{theorem} 
\begin{proof}
It is trivial to see that for a fixed $m \geq 2$ a sufficiently large $K_{n_1}$ exists such that on completion of the first Marcello global iteration the graph $K_n + \mathfrak{N}_m$ is complete. The reasoning follows along the line of that found in the proof of Proposition \ref{prop2}. Obviously the minimization of such $n_1$ to find a minimum $n$ is possible.
\end{proof}
\begin{theorem}\label{thm4}
For $\mathfrak{N}_m$, $m \geq 2$ there exists a minimum $n \geq 2$ such that $\varpi(K_n \cup \mathfrak{N}_m) = 2$. 
\end{theorem} 
\begin{proof}
It is trivial to see that for a fixed $m \geq 2$ a sufficiently large $K_{n_1}$ exists such that on completion of the first Marcello global iteration on $G_0 = K_n \cup \mathfrak{N}_m$ an important condition is met. That is, that the respective vertex degrees of $G_1$ are sufficiently large to complete $G_1$ in the second global iteration. Obviously the minimization of such $n_1$ to find a minimum $n$ such that $\varpi(K_n \cup \mathfrak{N}_m) = 2$ is possible. 
\end{proof}
\section{On certain classical graphs}
Theorem \ref{thm2}, Claims \ref{clm2}, \ref{clm3} and Proposition \ref{prop2} motivate researching certain classical graphs. First, consider the following pairwise procedure (for brevity, $\mathcal{P}$-procedure) to calculate the sum of integers. Firstly, let the string of integers have an even number of terms:
\begin{center}
$1+2+3+4+5+6+7+8+9+10+11+12+13+14 =$

$\overbrace{(1+2)+(3+4)+(5+6)+(7+8)+(9+10)+(11+12)+(13+14)} =$

$(3+7)+(11+15)+(19+23)+(27) =$

$(10+26)+(42+27) =$

$\underbrace{(36+69)} =$

$105$.
\end{center}
The lines between the over- and under-brace are called \emph{braced iterations}. Note that in the example there are $t = 4$ braced iterations. The initial string of integers to be summed has length $\ell = 14$. Furthermore, $14 = 2^3 +6$. Note that $2$ has maximum exponent i.e. $ex = 3$ in this particular number theoretical expression of the number $14$. 

Secondly, let the string of integers have an odd number of terms:
\begin{center}
$1+2+3+4+5+6+7+8+9 =$

$\overbrace{(1+2)+(3+4)+(5+6)+(7+8)+(9)} =$

$(3+7)+(11+15)+(9) =$

$(10+26)+(9) =$

$\underbrace{(36+9)} =$

$45$.
\end{center}
Note that in the example there are $t = 4$ braced iterations. The initial string of integers to be summed has length $\ell = 9$. Furthermore, $9 = 2^3 +1$. Note that maximum exponent i.e. $ex = 3$ in this particular number theoretical expression of the number $9$. 
\begin{claim}\label{clm5}
For the summation of a string of integers with the length of the string (number of terms) equal to $\ell \geq 2$ the value of $\ell$ can be bounded by $2^{ex} + k \leq \ell \leq 2^{ex+1}$, $k \geq 1$ where $ex$ is the smallest non-negative integer. Finally, the $\mathcal{P}$-procedure will require $t= ex + 1$ braced iterations to complete summation.
\end{claim}
\textit{Motivation:} The claim follows from the fact that natural number $\mathbb{N} = \{1,2,3,\dots\}$ can be partitioned as,
\begin{center}
$\mathbb{N} = \{1,2\}\cup \{3,4\}\cup \{5,6,7,8\} \cup \cdots \cup \{2^{ex}+1,\dots,2^{ex}\}\cdots\}$.
\end{center}
For example, for $\ell = 2$, $2^0 + 1 = 2 \leq 2^1$ thus, $t = 1$. Also, for $\ell = 8$, $2^2 + 4 = 8 \leq 2^3$ thus, $t = 3$ and for for $\ell = 11$, $2^3 + 3 = 11 \leq 2^4$ thus, $t = 4$.
\subsection{Paths:}
Recall that a path of order $n \geq 2$ and denoted by $P_n$ has vertex set $V(P_n) = \{v_1, v_2, v_3,\dots, v_n\}$ with the edge set $E(P_n) = \{v_iv_{i+1}:i = 1,2,3,\dots,n-1\}$.
\begin{lemma}\label{lem1}
(a) $\varpi(P_2) = 0$, $\varpi(P_i) = 1$ for $i = 3,4,5,6$ and $\varpi(P_7) = 2$.
\end{lemma}
\begin{proof}
The results $\varpi(P_2) = 0$, $\varpi(P_i) = 1$ for $i = 3,4,5,6$ are trivial to verify. Note that in the Marcello completion of $P_6$ the vertex degree of each vertex has been exhausted. The aforesaid implies that, if in path $P_7$ the equivalent Marcello edges as for $P_6$ by just replacing $v_6$ with say, $v_7$ were utilized the resultant graph has the construction of a $K_6$ on vertices $v_1,v_2,v_3,v_4,v_5,v_7$ together with $v_6$ adjacent to $v_5$ and $v_7$. However, $v_6$ awaits its turn to be utilized in the next local iteration. Since $deg_{P_7}(v_6) = 2$ and $v_6$ is non-adjacent to four vertices, completeness cannot be attained. Therefore, $\varpi(P_7) = 2$.
\end{proof}
An alternative proof for the result $\varpi(P_7) = 2$ was proposed by an \emph{anonymous reviewer}. Readers will hopefully appreciate the alternative approach.
\begin{proof}
In $P_7$ there are two vertices of degree $1$ and five vertices of degree $2$. Hence, a maximum of twelve edges can be added during the first Marcello global iteration. Thus, a resultant graph of maximum size of $18$ can yield. Since $K_7$ has size equal to $21$ at least one more global iteration is required. It is easy to see that one global iteration suffices. Therefore $\varpi(P_7) = 2$.
\end{proof} 
Observe that if in the proof above the vertices $v_1,v_2,v_3,v_4,v_5,v_6$ are utilized then the resultant graph has the construction of a $K_6$ with $v_7$ as a pendant to $v_6$. The aforesaid is a useful resultant while $v_7$ awaits its turn to be utilized in the next local iteration.

Consider the graphs $H_1,H_2,H_3,\dots,H_k$. Link each of the pairs $(H_1,H_2)$; $(H_2,H_3)$; $(H_3,H_4)$; $\cdots$ $(H_{k-2},H_{k-1})$; $(H_{k-1},H_k)$ with an arbitrary edge. We say that the graphs have been \emph{pearled}. Such pearled graphs are denoted by $\tilde{H_1}\tilde{H_2}\tilde{H_3}\tilde{\cdots}\tilde{H_k}$.
\begin{corollary}\label{cor4}
From Theorem \ref{thm2} it follows immediately that
\begin{center}
$\varpi(\tilde{K_n}\tilde{K_m}) = 1$.
\end{center}
\end{corollary}
The paths $P_6$ and $P_3$ can be pearled such that $P_9 = \tilde{P_6}\tilde{P_3}$. In fact, permitting the path $P_1$ makes it possible find a pearled path,
\begin{center}
$P_n = \underbrace{\tilde{P_6}\tilde{P_6}\tilde{P_6}\tilde{\cdots}\tilde{P_6}}_{\lfloor \frac{n}{6}\rfloor-times}\tilde{P_k}$, $n \geq 7$ and $0 \leq k \leq 5$.
\end{center}
By applying Marcello's completion (Definition \ref{def1}) to each $P_6$ and to $P_k$ in \lq isolated\rq  fashion a pearled string of complete graphs i.e.
\begin{center}
$\underbrace{\tilde{K_6}\tilde{K_6}\tilde{K_6}\tilde{\cdots}\tilde{K_6}}_{\lfloor \frac{n}{6}\rfloor-times}\tilde{K_k}$
\end{center}
is obtained. The aforesaid is clearly a non-optimal first global iteration. Consider this as an initialization iteration. 
\begin{proposition}\label{prop4}
For the path $P_n$, $n \geq 7$ and let
\begin{center}
$\ell = \lfloor \frac{n}{6}\rfloor + (n - \lfloor \frac{n}{6}\rfloor) = 2^{ex} + k$.
\end{center}
Then, $\varpi(P_n) \leq ex + 2$.
\end{proposition}
\begin{proof}
Let $P_n = \underbrace{\tilde{P_6}\tilde{P_6}\tilde{P_6}\tilde{\cdots}\tilde{P_6}}_{\lfloor \frac{n}{6}\rfloor-times}\tilde{P_k}$, $n \geq 7$ and $0 \leq k \leq 5$.

Perform the initial non-optimal global iteration of Marcello's completion to obtain
\begin{center}
$G_1 = \underbrace{\tilde{K_6}\tilde{K_6}\tilde{K_6}\tilde{\cdots}\tilde{K_6}}_{\lfloor \frac{n}{6}\rfloor-times}\tilde{K_k}$, $n \geq 7$ and $0 \leq k \leq 5$.
\end{center}
Proceed by the $\mathcal{P}$-procedure to perform Marcello's completion in pairs. By Theorem \ref{thm2} each pair of complete graphs yields a larger complete graph in each global iteration. From Claim \ref{clm5} read together with Corollary \ref{cor4} completion of $P_n$ will be reached after $(ex+1)+1$ iterations. It follows that $\varpi(P_n) \leq ex+2$.
\end{proof}
Following from Claim \ref{clm2} it is useful to note that if graph $G$ of order $n$ contains a Hamilton path $P_n$ then $\varpi(G) \leq \varpi(P_n)$.

\subsection{Cycles:} Recall that a cycle of order $n \geq 3$ and denoted by $C_n$ has vertex set $V(C_n) = \{v_1, v_2, v_3,\dots, v_n\}$ with the edge set $E(C_n) = \{v_iv_{i+1}:i = 1,2,3,\dots,n-1\} \cup \{v_1v_n\}$.

A trivial bound is given by Claim \ref{clm2} i.e. $\varpi(C_n) \leq \varpi(P_n)$.
\begin{lemma}\label{lem2}
(a) $\varpi(C_3) = 0$, $\varpi(C_i) = 1$ for $i = 4,5,6,7$.
\end{lemma}
\begin{proof}
The proof for only $C_7$ will be provided. Add the Marcello edges $v^+_iv_{i+2}$, $v^+_iv_{i+3}$, $i = 1,2,3,4,5,6,7$ (normal modulo $7$ count applies). It is easy to see that the resultant graph has $21$ distinct edges. Hence, the cycle $C_7$ reached completion. 
\end{proof} 
By prohibiting the edges $v^+_1v_n$ and $v^+_nv_1$ a cycle may initially be dealt with in terms of its spanning subgraph $P_n$. By Claim \ref{clm2}, $\varpi(C_n) \leq \varpi(P_n)$. An analysis similar to that found in sub-section 3.1 reveals that the existence of the edge $v_1v_n$ is not always sufficient to reduce the number of global iterations required to complete a cycle $C_n$ vis-$\grave{a}$-vis completing $P_n$. 
\begin{claim}
(a) A Hamiltonian graph $G$ has $\varpi(G) \leq \varpi(C_n)$.

(b) A star, $S_{1,n}$ has $\varpi(S_{1,n}) \leq \varpi(C_n) + 1$.

(c) A wheel, $W_{1,n}$ has $\varpi(W_{1,n}) \leq \varpi(C_n)$.

(d) The Petersen graph $\mathcal{P}$ has $\varpi(\mathcal{P}) = 1$.
\end{claim}
\textit{Motivation:}

(a) Since a Hamiltonian graph of order $n$ is chorded cycle on $n$ vertices the cycle serves as a spanning subgraph. The result follows from Claim \ref{clm2}.

(b),(c) Clearly after the first global iteration, $G_1 \cong W_{1,n}$ where, $G = S_{1,n}$. Only the rim (or subgraph $C_n$) of the wheel requires Marcello's completion. The rim vertices have the advantage of degree $3$.
\begin{proof}
(d) Consider the classical figure of the Petersen graph. For the Petersen graph $\mathcal{P}$ label the vertices of the outer $C_5$ as $v_1,v_2,v_3,v_4,v_5$. The corresponding inner vertices are labeled $u_1,u_2,u_3,u_4,u_5$. Begin with the Marcello edges $v_1^+v_3$, $v^+_2v_4$, $v^+_3v_5$, $v^+_4v_1$, $v^+_5v_2$ and $u^+_1u_2$, $u^+_2u_3$, $u^+_3u_4$, $u^+_4u_5$, $u^+_5u_1$. Note that each vertex remains with a $+2$ degree-count each. The aforesaid may be utilized during exhaustion of the first global iteration. The partial progress made is a graph consisting of two distinct $K_5$ graphs which are linked by the set of edges $X = \{v_1u_1, v_2u_2, v_3u_3, v_4u_4, v_5u_5\}$. Consider the induced edge-subgraph $\langle X\rangle$. Now add the edges $v^+_1u_2$, $v^+_1u_3$, $v^+_2u_3$, $v^+_2u_4$, $v^+_3u_4$, $v^+_3u_5$, $v^+_4u_5$, $v^+_4u_1$, $v^+_5u_1$, $v^+_5u_2$ as well as, $u^+_1v_2$, $u^+_1v_3$, $u^+_2v_3$, $u^+_2v_4$, $u^+_3v_4$, $u^+_3v_5$, $u^+_4v_5$, $u^+_4v_1$, $u^+_5v_1$, $u^+_5v_2$.

The first Marcello global iteration is complete and the Petersen graph reached completion yielding $K_{10}$.
\end{proof}
\section{Research avenues}
\textbf{Problem 1:} Let $G$ be a graph with $\varpi(G) = k$. In Marcello's completion (see Definition \ref{def1}) a subset $X$ of $V(G)$ is defined. Can a subset $X'\subseteq X$ with smallest cardinality be selected such that the iterative application of Marcello's completion in respect of $X'$ yields $\varpi(G) = k$?

\textbf{Problem 2:} From Claim \ref{clm1} it can be said that Marcello's completion is a \emph{self learning algorithm}. Question is: For a non-complete and non-null graph $G$ of order $n$ consider all possible minimal Marcello sequences. Determine up to isomorphism, the (maximum) number of distinct graphs for which the Marcello sequences reveal the respective Marcello numbers? The said number of graphs is called the \emph{Marcello index} of $G$ and is denoted by $\varpi_i(G)$. It is easy to see that for a graph $G$ with $\varpi(G) = k$ the Marcello index is bounded by $\varpi_i(G) \geq \varpi(G) - 1$.

For example, consider the graph $G = P_2 \cup 3K_1$. It is easy to verify that up to isomorphism the first global iteration of Marcello's completion results is either $K_3 \cup 2K_1$ or $P_4 \cup K_1$. In the second global iteration of Marcello's completion and up to isomorphisms the graphs, either $K_3 + 2K_1$ or $K_2 + 3K_1$ or $\overline{P_3 \cup 2K_1}$ result. On the third global iteration the complete graph $K_5$ yields. Therefore, $\varpi(G) = 3$ whilst the \emph{self learned} results are $\varpi(K_3 \cup 2K_1) = \varpi(P_4 \cup K_1) = 2$ and $\varpi(K_3 + 2K_1) = \varpi(K_2 + 3K_1) = \varpi(\overline{P_3 \cup 2K_1}) = 1$. Thus $\varpi_i(G) = 5$. Intuitively one might think that if $\varpi(G) = 1$ then $\varpi_i(G) = 0$. The latter is not correct as can be seen from Figures $\ref{fig1}$, $\ref{fig2}$ and $\ref{fig3}$. 

\textbf{Problem 3:} For $n \in \mathbb{N}$ and up to isomorphism the number of distinct non-complete and non-null graphs of order $n$ is finite, say $q \in \mathbb{N}$. Find a set of such graphs, say $X = \{G_1, G_2, G_3,\dots, G_t\}$ which collectively through the respective Marcello sequences reveal the Marcello numbers of all such graphs of order $n$. The aforesaid implies that $t + \sum\limits_{G_j \in X}\varpi_i(G_j) \geq q$.

For example, since it is known that up to isomorphism eleven distinct simple graphs exist on four vertices, nine of the graphs are non-complete and non-null. It is easy to verify that the graphs $P_2 \cup 2K_1$, $P_3 \cup K_1$, $2P_2$ and $K_{1,3}$ have sufficient Marcello sequences to collectively reveal the Marcello numbers of all the non-complete and non-non-null graphs on four vertices. The next lemma could assist further research.
\begin{lemma}\label{lem3}
Consider a non-complete and non-null graph $G$ of order $n \geq 3$. For any completed global iteration of Marcello's completion:

(a) A connected $G$ cannot yield a disconnected graph $G_i$.

(b) $G \ncong S_{1,n-1}$ cannot yield the star $S_{1,n-1}$ (or, $K_{1,n-1}$).
\end{lemma}
\begin{proof}
(a) Trivial since Marcello's completion strictly add edges.

(b) Assume $G \ncong S_{1,n-1}$. The \lq most\rq  disconnected (i.e. least edges in $G$) is $P_2 \cup (n-2)K_1$. After the first global iteration of Marcello's completion a subgraph either $K_3$ or $P_4$ exists in $G_1$. Since neither of these subgraphs is a subgraph of a star, a star cannot yield. Clearly no other graph on $n \geq 3$ vertices can contradict the aforesaid conclusion. That settles the result.
\end{proof}
Lemma \ref{lem3} suggests that for some minimum subset $Y$ of all the disconnected graphs on $n \geq 3$ vertices the set $Y \cup \{S_{1,n-1}\}$ will reveal $\varpi(G)$ for all $G$ on $n$ vertices. 
\section{Conclusion}
With regards to Figures 1, 2, 3 it appears to be a hard problem for graphs in general to find a maximum number of Marcello edges during a global iteration. Only maximality of the number of Marcello edges is guaranteed by Definition \ref{def1}. In addition to the said difficulty it must be noted that for an arbitrary large $k \in \mathbb{N}$ there exist a graph $G$ with $\varpi(G) > k$. Amongst others, it follows from the fact that if $n$ is arbitrary but fixed then,
\begin{center}
$\lim\limits_{m\to \infty}\varpi(P_n\cup \mathfrak{N}_m) = \infty$.
\end{center} 
The problem of finding a method to guarantee finding the maximum number of Marcello edges remains open. Author views the latter problem as challenging. Solving this problem could yield closed results for numerous graphs.
\begin{conjecture} 

(a) A non-complete, connected graph $G$ with at most one pendent vertex and which has
\begin{center}
$\sum\limits_{0<deg_G(v_i) < n-1}deg_G(v_i) + |E(G)| > \frac{n(n-1)}{2}$
\end{center}
has $\varpi(G) = 1$.

(b) A non-complete, connected graph $G$ with at most one pendent vertex and which has
\begin{center}
$\sum\limits_{0<deg_G(v_i) < n-1}deg_G(v_i) + |E(G)| = \frac{n(n-1)}{2}$
\end{center}
has $\varpi(G) = 1$ or $2$.
\end{conjecture}
Research is underway with regards to graphs with specific characteristics such as circulants, regular graphs, trees, block graphs and so on. The use of clique partitioning and other clique decomposition techniques is also under research. Interested readers can refer to \cite{bhasker, cavers}.
\section*{Dedication}
This paper is dedicated to Marcello Marques. Marcello is a baby boy the author was blessed to meet during a motorcycle adventure ride in August 2024. Marcello has \lq X\rq-tra $C21$ needs and I pray that this paper and further research by others, will add to the purpose of his life.
\section*{Acknowledgment}
The author would like to thank the anonymous referees for their constructive comments which helped to improve on the elegance of this paper.
\subsection*{Conflict of interest:} The author declares there is no conflict of interest in respect of this research. Furthermore, the author declares that no generative AI has been utilized.

\end{document}